\documentclass[letterpaper, reqno, 11pt]{amsart}
\usepackage[margin=1.0in]{geometry}
\usepackage{amsmath,amssymb,amsthm}
\usepackage{graphicx}
\graphicspath{{images_cropped/} }

\usepackage[percent]{overpic}

\usepackage{color}

\newtheorem{thm}{Theorem}[section]
\newtheorem{lem}{Lemma}[section]
\newtheorem{defn}{Definition}[section]
\newtheorem{conjecture}{Conjecture}[section]

\theoremstyle{remark}

\numberwithin{equation}{section}

\begin{document}
 \title{On the polynomial Wolff axioms}
\author{Nets Hawk Katz}
\address{California Institute of Technology, Pasadena CA, U.S.A.}
\email{nets@caltech.edu}
\thanks{Supported by NSF grant DMS 1565904 and by MINECO grants SEV-2015-0554 and MTM2017-85934-C3-1-P}

\author{Keith M. Rogers}

\address{Instituto de Ciencias Matem\'aticas CSIC-UAM-UC3M-UCM, Madrid, Spain}
\email{keith.rogers@icmat.es}

\begin{abstract}
We confirm a conjecture of Guth concerning the maximal number of $\delta$-tubes, with $\delta$-separated directions, contained in the $\delta$-neighborhood of a real algebraic variety. Modulo a factor of $\delta^{-\varepsilon}$,   we also prove  Guth and Zahl's generalized version for semialgebraic sets. Although the applications are to be found in harmonic analysis, the proof will employ deep results from algebraic and differential geometry, including Tarski's projection theorem and Gromov's algebraic lemma.
\end{abstract}

\maketitle

\section{Introduction}\label{introSection}

For $\delta>0$, we consider $\delta$-neighborhoods of unit line segments, arbitrarily positioned  in a compact subset of Euclidean space. We call these $\delta$-tubes and, to avoid introducing an extra parameter, they will be supported in a ball of radius two from now on. A formulation of the Kakeya conjecture seeks to estimate the number of $\delta$-tubes, pointing in $\delta$-separated directions, in terms of the Lebesgue measure of any set that contains them (take $\lambda=1$ in inequality \eqref{kakeq} below). 

Guth and Zahl showed how progress can be made via polynomial partitioning  \cite{GZ}, a technique introduced in \cite{GK}. This partitions the underlying space with the zero set of a polynomial of degree~$D$, after which a line cannot intersect more than $D+1$ of the resulting subsets. The problem is typically reduced to the harder case concerning what happens on or near the zero set. For this it is useful to know how many $\delta$-tubes, pointing in $\delta$-separated directions, can be contained in the $\delta$-neighborhood of the zero set. 
In relation to the closely related Fourier restriction problem, Guth made the following conjecture \cite[pp. 49]{G}.

\begin{conjecture}\label{guth} For all integers $n,D\ge 2$ and all $\varepsilon>0$, there is a constant $C(n,D,\varepsilon)>0$ so that the number of $\delta$-tubes, pointing in $\delta$-separated directions, contained in the $\delta$-neighborhood of an $m$-dimensional algebraic variety  $Z\subset \mathbb{R}^{n}$, of degree at most~$D$, is bounded by $C(n,D,\varepsilon)\delta^{1-m-\varepsilon}$. 
\end{conjecture}
This was proven by Guth,  with $n=3$, yielding progress on the three-dimensional restriction conjecture~\cite{G1, H},  and  by Zahl, with $n=4$, yielding progress on the four-dimensional Kakeya conjecture~\cite{Z}, as well as the four-dimensional restriction conjecture~\cite{D}. Here we will prove Conjecture~\ref{guth} in all dimensions. This implies that $\delta$-neighborhoods of varieties cannot contradict the Kakeya conjecture. This should be compared with the {\it grains} decomposition of \cite{GZ} which tells us that a union of $\delta$-tubes can only have small measure if it has some algebraic structure.
On the other hand, in \cite{G} it is noted that a resolution of Conjecture~\ref{guth} would lead to further improvements for the restriction conjecture in higher dimensions. Indeed, the $k$-broad estimates of \cite{G} can be improved by arguing as in \cite{G1}, mapping from $L^\infty$ rather than $L^2$ so as to take advantage of Conjecture~\ref{guth} with $m=n-1$. This controls the $k$-broad norm with Lebesgue exponent $p=\frac{2n}{n-1}\frac{n(n+k)-k}{n(n+k)-n}$ and so the estimate with $k=(n+1)/2$ can be inputted into Proposition~9.1 from \cite{G} to provide an improved adjoint restriction estimate in the range $p>\frac{2n}{n-1}(1+\frac{n-1}{(3n-1)n})$ in odd dimensions. 

More generally, Guth and Zahl considered the following definition. In Wolff's original version, the semialgebraic sets $S$ are taken to be truncated $\delta$-neighborhoods of $2$-planes~\cite{W}. It is clear that $\delta$-tubes, pointing in $\delta$-separated directions, satisfy the Wolff axioms, and Guth and Zahl conjectured that they also satisfy the following stronger condition; see \cite[pp. 4]{GZ}.
\begin{defn}\label{defnTubeWolffAxioms} 
We say that sets $\mathbb{T}$ of $\delta$-tubes in $\mathbb{R}^n$ satisfy the polynomial Wolff axioms if, for every integer $E\ge 2$, there is a constant $C(n,E)>0$ so that
$$\#\big(\big\{ T \in \mathbb{T} :  |T \cap S| \geq \lambda |T| \big\}\big) \leq C(n,E) |S| \delta^{1-n} \lambda^{-n}$$
whenever $S$ is a semialgebraic set, of complexity at most $E$, and $\lambda\ge \delta>0$.
\end{defn}

We will prove the following theorem, confirming their conjecture up to a factor of $C_\varepsilon\delta^{-\varepsilon}$. The tubes are contained in a ball of $\mathbb{R}^n$, and the intersection of this with the $\delta$-neighborhood of an $m$-dimensional variety, of degree at most $D$, forms a semialgebraic set $S$ with complexity bounded in terms of~$n$ and $D$. Moreover, by Wongkew's lemma~\cite{Wo}, the measure of such an $S$ is bounded by $c(n,D)\delta^{n-m}$. Thus, Conjecture~\ref{guth} is proved  by taking $\lambda=1$ in the following theorem.

\begin{thm}\label{mainThm}
Let~$n,E\ge 2$ be integers and $\varepsilon>0$. Then there is a constant $C(n,E,\varepsilon)>0$ so that, for every set $\mathbb{T}$ of $\delta$-tubes in $\mathbb{R}^n$, pointing in  $\delta$-separated directions, 
\begin{equation}\label{kakeq}\#\big(\big\{ T \in \mathbb{T} :  |T \cap S| \geq \lambda |T| \big\}\big) \leq C(n,E,\varepsilon) |S| \delta^{1-n-\varepsilon} \lambda^{-n}\end{equation}
whenever $S$ is a semialgebraic set, of complexity at most $E$, and $\lambda\ge \delta>0$.
\end{thm}

The proof 
will employ deep results from both algebraic and differential geometry. On the one hand, we use quantifier elimination to build new semialgebraic sets with bounded complexity from known semialgebraic sets.  On the other hand, we will use Gromov's algebraic lemma to nicely parametrize semialgebraic sets of bounded complexity. 
Rather than apply Gromov's lemma to~$S$, we would like to apply it to our set of tubes, however this is not semialgebraic. We replace it with a semialgebraic version using quantifier elimination. We first consider all the tubes contained in $S$ and then take a semialgebraic section, the result being that the tubes are repositioned. This will be discussed in more detail in the following section.
 
 In the third section, we prove a simplified version of  Theorem~\ref{mainThm}, where the intersections of the tubes with $S$ contain truncated $\delta$-tubes of length $\lambda$, the advantage being that we can perform the previous steps to obtain a semialgebraic set of tubes. We bound $|S|$ below by the measure of the union of these tubes, each slice of which can be written in terms of the parametrisation given by Gromov's lemma. One part of the parametrisation maps into the directions (and there is no more than one tube for each direction), and the other part into the uncontrolled position of the tube. We approximate the parametrisation by a polynomial, allowing us to apply B\'ezout's theorem, in order to ensure that this uncontrolled part does not interfere too often.
 
In the final section, we complete the proof. This involves a further application of B\'ezout's theorem, a change of scales, and dyadic pigeonholing in order to obtain a version of \eqref{kakeq} with~$|S|$ on the right-hand side replaced by the measure of the $\delta^n$-neighborhood of $S$. The proof is then completed by bounding $|S_{\delta^n}|=|S_{\delta^n}\backslash S|+|S|$ by a constant multiple of $|S|$, an easy consequence of the Milnor--Thom theorem \cite{M} combined with Wongkew's lemma \cite{W}.

\vspace{0.5em}

\noindent{\bf Acknowledgements:} The first author would like to thank Josh Zahl for helpful discussions. In particular the proof of Lemma \ref{Extractingsection} came
from a conversation with him. The second author would like to thank Jonathan Hickman for helpful discussions regarding the application to restriction.

\section{Semialgebraic sets, quantifier elimination and Gromov's algebraic lemma}

Following \cite{BPR}, we say that the semialgebraic sets of $\mathbb{R}^n$ are the smallest family of sets, closed under finite unions, intersections, and complements, that contains both $\{x: P(x)=0\}$ and $\{x: Q(x) > 0\}$ for all  polynomials $P$ and $Q$.
We say that the complexity of a semialgebraic set is the smallest sum of the degrees of the polynomials appearing in a complete description of the set.

Perhaps the most fundamental result in the subject of semialgebraic sets is Tarski's projection theorem; see for example \cite{BPR}.

\begin{thm}[Tarski] \label{Tarski}
Let $\Pi $ be the orthogonal projection of $\mathbb{R}^n$ into its first $n-1$ coordinates. Then for every $E\ge 1$, there is a constant $C(n,E) > 0$ so that, for every  semialgebraic $S \subset \mathbb{R}^n$ of complexity at most $E$, the projection $\Pi (S)$ has complexity at most $C(n,E)$. 
\end{thm}

Noting that $(x_1,\dots,x_{n-1}) \in \Pi(S)$ if and only if $\exists\ (x_1,\dots,x_{n-1},x_n) \in S$, we associate this theorem with quantifiers.  Roughly speaking, Tarski's theorem tells us that any set described using semialgebraic sets and quantifiers is semialgebraic with complexity depending
only on the length of the description and the complexity of the semialgebraic sets used in the description.

For notational convenience we work in  $\mathbb{R}^{n+1}$ rather than $\mathbb{R}^{n}$. With $\lambda\ge\delta$ and $t_0\in[-2,2]$, we will consider truncated $\lambda\times \delta$-tubes in $\mathbb{R}^{n+1}$ defined by
$$T_{\mathbf{a},\mathbf{d}} = \big\{ (\mathbf{x}, t)\in \mathbb{R}^{n}\times[t_0,t_0+\lambda]:  |\mathbf{x} - \mathbf{a} -t \mathbf{d}| \leq \delta\big\},\qquad (\mathbf{a},\mathbf{d})\in[0,1]^{2n}.$$
Note that $T_{\mathbf{a},\mathbf{d}}$ is a semialgebraic set of fairly small complexity.

\begin{lem} \label{We'reinbusiness} Let $S \subset \mathbb{R}^{n+1}$ be a semialgebraic set of complexity at most $E$. Then
$$L_S := \big\{ (\mathbf{a},\mathbf{d})\in[0,1]^{2n}: T_{\mathbf{a},\mathbf{d}} \subset S \big\}$$
is a semialgebraic set of of complexity at most $C(n,E)$, a constant depending only on~$n$ and~$E$.
\end{lem}

\begin{proof} This is an immediate consequence of Tarski's projection theorem. First  we  write $$L_S= \big\{ (\mathbf{a},\mathbf{d})\in[0,1]^{2n}\,:\,  (\mathbf{x},t)\in S\ \forall\ (\mathbf{x},t)\in T_{\mathbf{a},\mathbf{d}}\big\}.$$
We then define the clearly semialgebraic $Y$ by
$$Y =\big\{(\mathbf{a},\mathbf{d},\mathbf{x},t)\in [0,1]^{2n}\times\mathbb{R}^n\times[t_0,t_0+\lambda]\,:\, (\mathbf{x},t)  \notin S, \ (\mathbf{x},t)\in T_{\mathbf{a},\mathbf{d}} \big\}.$$
Writing 
$Z=\Pi (Y)$, where $\Pi$ is the projection $(\mathbf{a},\mathbf{d}, \mathbf{x}, t)\mapsto(\mathbf{a}, \mathbf{d})$, by Theorem~\ref{Tarski} we conclude that~$Z$ is semialgebraic of complexity depending only on~$n$ and $E$. The proof is completed by noting that $L_S$ is the complement of $Z$ in $[0,1]^{2n}$.
\end{proof}

Noting that $L_S$ is closed if $S$ is closed, we have shown that given a closed semialgebraic set, the set of tubes it contains is closed and semialgebraic. Next we will show that we can extract a section semialgebraically. That is, we can choose one tube for each direction.

\begin{lem} \label{Extractingsection} Let $S \subset \mathbb{R}^{2n}$ be a compact semialgebraic set of complexity at most $E$. Let~$\Pi$ be the orthogonal projection into the final $n$ coordinates 
$(\mathbf{a},\mathbf{d}) \mapsto \mathbf{d}.$
Then there is a constant $C(n,E)>0$, depending only on $n$ and $E$, and a semialgebraic set $Z$, of complexity at most~$C(n,E)$, so that
$$Z \subset S,\quad\quad
\Pi(Z)= \Pi(S),$$
and so that for each $\mathbf{d},$ there is at most one $\mathbf{a}$ with
$(\mathbf{a},\mathbf{d}) \in Z.$
\end{lem}

\begin{proof} It suffices to show that for the projection $\Pi_1$ defined by
$(\mathbf{a},\mathbf{d})\mapsto(a_2,\dots, a_n,\mathbf{d}),$
there is a constant $C(E)>0$ and a semialgebraic $Z_1$ of complexity at most $C(E)$,  so that
$$Z_1 \subset S,\quad\quad
\Pi_1(Z_1)=\Pi_1(S),$$
and so that for any $(a_2,\dots,a_n,\mathbf{d})$ there is at most one $a_1$ with $(\mathbf{a},\mathbf{d}) \in Z_1.$
Having done that, we obtain $Z_2$ by applying the same result to $Z_1$ with the first coordinate replaced by the second, obtain $Z_j$ from $Z_{j-1}$ with the first coordinate replaced
by the $j$th, and finally setting $Z=Z_n$.

It suffices to see that there is a semialgebraic choice of $Z_1$. Whenever $(a_2,\dots, a_n,\mathbf{d}) \in \Pi_1(S)$ we let $(\mathbf{a},\mathbf{d}) \in Z_1$ for $a_1$ the maximal value so that
$(\mathbf{a},\mathbf{d}) \in S$. More logically, we write
$$Z_1=\big\{(\mathbf{a},\mathbf{d}) \in S\,:\, x \leq a_1\ \forall\ (x,a_2, \dots,a_n,\mathbf{d}) \in S \big\}.$$
As before we introduce the clearly semialgebraic $Y$ defined by
$$Y=\big\{(x, \mathbf{a},\mathbf{d})\in \mathbb{R}\times S\,:\, x>a_1,\ (x,a_2, \dots,a_n,\mathbf{d}) \in S \big\},$$
and use Theorem~\ref{Tarski} to project $Y$ to its last $2n$ coordinates. We then recover $Z_1$, by taking the complement in $S$, to complete the proof.
\end{proof}

An elementary proof of the following algebraic lemma can be found in the work of Burguet~\cite{B}.

\begin{lem}[Gromov] \label{Burguet} For all integers $d, E, r \geq 1$, there exists $M(d, E, r) < \infty$ with the following
properties. For any compact semialgebraic  set $A \subset [0,1]^d$, of dimension $n$ and complexity at most~$E$, there exists
an integer $N\le M(E,d,r)$ and maps $\phi_1,\dots,\phi_N: [0,1]^n \longrightarrow [0,1]^d$ so that
$$\bigcup_{j=1}^N  \phi_j ([0,1]^n) =A\quad \text{and}\quad \|\phi_j\|_{C^r}:=\max_{|\alpha|\le r}\|\partial^\alpha\phi_j\|_{\infty} \leq 1.$$
\end{lem}

A weaker version of this was first proved by Yomdin \cite{Y}. It was first stated as presented here by Gromov.  The first detailed proof of this version appears to have been given by Pila and Wilkie~\cite{PW}.

\section{Proof of Theorem~\ref{mainThm} with $\lambda=1$}

For notational convenience we work in $\mathbb{R}^{n+1}$ rather than $\mathbb{R}^n$. As we can suppose that the $\delta$-tubes are contained in $B(0,2)$,  without loss of generality we can suppose that our semialgebraic sets $S\subset \mathbb{R}^{n+1}$ are compact.   We choose our coordinates so that a large proportion (at least a fraction~$1/4^{n}$) of our tubes have central lines segments that can be written as $(\mathbf{a},0)+t(\mathbf{d},1)$ with $t$ in an interval $I\subset[-2,2]$ and~$\mathbf{d}\in[0,1]^n$. Similarly, by translation if necessary,  we can also suppose that $\mathbf{a}\in[0,1]^n$. 

Recalling from the previous section that a $\lambda\times \delta$-tube is defined to be of the form
$$T_{\mathbf{a},\mathbf{d}}(\lambda,\delta) = \big\{ (\mathbf{x}, t)\in \mathbb{R}^{n}\times[t_0,t_0+\lambda]:  |\mathbf{x} - \mathbf{a} -t \mathbf{d}| \leq \delta\big\},\qquad (\mathbf{a},\mathbf{d})\in[0,1]^{2n},$$
our $\delta$-tubes always contain a $\lambda\times\delta$-tube with $\lambda=\tfrac{1}{2}(n+1)^{-1/2}$. Thus, to count the number of $\delta$-tubes entirely contained in $S$, it will suffice to prove the following theorem. This does not yet complete the proof of Theorem~\ref{mainThm}, as the intersection of $S$ with a $\delta$-tube need not contain a $\lambda\times \delta$-tube, for any $\lambda$, when the $\delta$-tube is not contained in $S$. In that case the mass  can be distributed along the length of the whole tube.

\begin{thm} \label{almostsimple}
Let~$n,E\ge 1$ be integers and $\varepsilon>0$. Then there is a constant $C(n,E,\varepsilon)>0$ so that, for every set $\mathbb{T}$ of $\lambda\times\delta$-tubes in $\mathbb{R}^{n+1}$, pointing in  $\delta$-separated directions, 
\begin{equation}\label{croc}\#\big(\big\{ T \in \mathbb{T} :  T\subset S \big\}\big) \leq C(n,E,\varepsilon) |S| \delta^{-n-\varepsilon} \lambda^{-n-1}\end{equation}
whenever $S$ is a semialgebraic set, of complexity at most $E$, and $\lambda\ge \delta>0$.
\end{thm}

\begin{proof}  We first cover the $t$-interval $[-2,2]$ with nonoverlapping intervals $I_k$ of length $\lambda/2$. The projection of each $T \in \mathbb{T}$ into the $(n+1)$th coordinate must contain some $I_k$. For each $T$, we choose such a $k$ and declare that $T \in \mathbb{T}_k$. We let 
$S_k$ be the subset of $S$ consisting of points whose $(n+1)$th coordinate is in $I_{k-1}\cup I_k\cup I_{k+1}$. Then it suffices to prove
\begin{equation*}\#\big(\big\{ T \in \mathbb{T}_k : T \subset S_k \big\}\big) \leq C(n,E,\varepsilon) |S_k| \delta^{-n-\varepsilon} \lambda^{-n-1}.\end{equation*}
Relabelling $\mathbb{T}_k, S_k$ and $I_k$ by $\mathbb{T}, S$ and $I$,  for the sake of a contradiction we assume that for all $C>0$, we can find sets $\mathbb{T}$ of $\lambda\times\delta$-tubes, pointing in  $\delta$-separated directions, and semialgebraic sets~$S$, of complexity bounded by $E$, such that
\begin{equation}\label{reptile}\#\big(\big\{ T \in \mathbb{T} :  T\subset S \big\}\big) > C|S| \delta^{-n-\varepsilon}\lambda^{-n-1}\end{equation}
for some $\lambda\ge \delta>0$. We can suppose that $|S|\ge \lambda\delta^{n}$ as otherwise $S$ would not contain a single tube. Note also that \eqref{croc} clearly holds when restricted to all $\delta>c>0$, by simply taking $C(n,E,\varepsilon)$ sufficiently large. Thus the $\delta$ for which \eqref{reptile} holds must tend to zero as $C$ tends to infinity. 

Now instead of counting the tubes of $\mathbb{T}$ directly, we first consider $L$ defined by $$L = \big\{ (\mathbf{a},\mathbf{d})\in[0,1]^{2n}: T_{\mathbf{a},\mathbf{d}}(\lambda, \delta/2) \subset S \big\},$$  the advantage being that we can apply Lemma~\ref{We'reinbusiness} to see that $L$ is  semialgebraic. Moreover, applying Lemma \ref{Extractingsection} to $L$, we obtain a semialgebraic section~$L'$ consisting of a single vector $(\bf{a},\bf{d})$ for each $\mathbf{d}$ appearing in  $L$.
Letting $\Pi$ denote the projection $(\mathbf{a},\mathbf{d})\mapsto \mathbf{d}$, we then have 
$$|\Pi(L')| > C|S| \delta^{-\varepsilon}\lambda^{-n-1}.$$
This is because for each $T_{\mathbf{a},\mathbf{d}}(\lambda, \delta)\in\mathbb{T}$ there is a whole $n$-dimensional ball $B(\mathbf{d},\delta/2)$ in $\Pi(L')$, and these balls are disjoint due to the fact that the directions of $\mathbb{T}$ are $\delta$-separated. 
Given that~$L'$ can be considered to be the graph of a function that maps from $\Pi(L')\subset [0,1]^n$, we see that~$L'$ is an $n$-dimensional subset of $[0,1]^{2n}$.

We apply Gromov's algebraic lemma,  Lemma \ref{Burguet}, to $L'$ with $r$ taken to be the first integer larger than $4n^2/\varepsilon$. This breaks $L'$ into $N$ pieces, with $N$ depending only on $n$, $E$ and $r$. For each piece~$L_j$,
there is a map 
$(F_j,G_j) : [0,1]^n \longrightarrow [0,1]^{2n},$
with
$$(F_j,G_j)([0,1]^n) = L_j\qquad\text{and}\qquad\|(F_j,G_j)\|_{C^{r}} \le 1.$$
By the pigeonhole principle, there is a choice of $j$ for which
\begin{equation*}|G_j([0,1]^n)|=| \Pi ( L_j) | >  C|S| \delta^{-\varepsilon}\lambda^{-n-1}.\end{equation*}
Moreover, we can find a ball $B\subset [0,1]^n$, centered at $\mathbf{x}_0$ and of diameter $\delta^{\frac{\varepsilon}{2n}}$, so that
\begin{equation}\label{annette}|G_j(B)| >  C|S| \delta^{-\varepsilon/2}\lambda^{-n-1}.\end{equation} 
We use this large set of directions  to find a lower bound on $|S|$ that will yield the contradiction.

First we replace $(F_j,G_j)$ by $(F,G)$, the $(r-1)$th degree  Taylor approximation of $(F_j,G_j)$ at~${\bf x}_0$. By the estimates on the $C^r$ norm of $(F_j,G_j)$, given by the Gromov algebraic lemma, we have
$$|(F_j,G_j)({\bf x})- (F,G)({\bf x})| \le |{\bf x}-{\bf x}_0|^r.$$
As we chose $r>4n^2/\varepsilon$, for $\mathbf{x}$ in our small ball $B$ centered at $\mathbf{x}_0$, this yields 
\begin{equation}\label{retro}|(F_j,G_j)({\bf x}) - (F,G)({\bf x})| \le \tfrac{1}{4}\delta^{2n}.\end{equation}
We see that continuous $G$ maps $B$ into the $\delta^{2n}$-neighborhood of~$G_j(B)$, and in particular the boundary of $B$ maps into the $\delta^{2n}$-neighborhood of the boundary of~$G_j(B)$. By \eqref{annette}, recalling that $|S|\ge \lambda\delta^{n}$, we can conclude that
\begin{equation}\label{tit}|G(B)| > C|S|\lambda^{-n-1}\end{equation}
whenever $C$ is sufficiently large so that $\delta^{2n}$ is sufficiently small.  If there are points $\mathbf{x}\in B$ for which the determinant of the Jacobian matrix $DG(\mathbf{x})$ is zero, they are mapped to a null set, by Sard's theorem. Thus we can remove them without affecting the validity of \eqref{tit}.

Now, by \eqref{retro} and the fact that the tubes are contained in $S$,  we have
 $(F({\bf x}) + t G({\bf x}),t) \in S$ for all $t\in I$ and~${\bf x} \in B$. Thus we can estimate
\begin{equation*}\label{rely}|S|\ge \int_I |( F + tG) (B)|\, dt.\end{equation*}
In order to contradict \eqref{tit}, we would like to bound this below by $\lambda^{n+1}|G(B)|$. Using the change of variables formula, this will follow from estimates for the Jacobian determinants, however
for fixed~$t$, it is not necessarily the case that $ F + tG$ is one-to-one. For this reason, we prepared a substitute, namely that $ F + tG$ is a polynomial in $n$ variables of degree $r-1$. By restricting ourselves to $B_t \subset B$,
defined to be the points ${\bf x}\in B$ where $(DF + tDG) ({\bf x})$ is invertible, the values of $F + tG$ are isolated, even after complexifying $F$ and $G$.
Thus, by B\'ezout's theorem, we see that $F + tG$ maps at most $(r-1)^n$ points of $B_t$ to the same place. Partitioning $B_t$ into sets $U_k$ on which $F+tG$ is one-to-one, by the change of variables formula, we obtain
$$
|(F + tG)(B_t)|\ge \frac{1}{(r-1)^n}\sum_{k}|(F + tG)(U_k)|=\frac{1}{(r-1)^n}\sum_{k}\int_{U_k}|(DF + tDG)(\mathbf{x})|\, d\mathbf{x}, 
$$
where $|(DF + tDG)(\mathbf{x})|$ denotes the absolute value of the determinant. Summing up and integrating in $t$, this yields
\begin{equation}\label{annette2}|S| \ge\frac{1}{(r-1)^n} \int_{I}\int_{B_t} | ( DF + tDG)({\bf x}) |\, d{\bf x}dt.\end{equation}
Note that there may be values of $t\in I$ for which $B_t$ is the empty set, however we will see that this cannot happen too often.

It remains to bound  $| ( DF + tDG) ({\bf x})|$ from  below in terms of $| DG({\bf x})|$. In order to do this,  we  first note  that
$$|  (DF + tDG) ({\bf x})| = |P_{\bf x}(t)|,$$
where $P_{\bf x}(t)$ is a polynomial of degree $n$.  Fixing ${\bf x}$ for the moment, we write
$$P_{\bf x}(t)=|DG({\bf x})|(t-r_1)(t-r_2) \dots (t - r_n),$$
where $r_1,\dots r_n$ may be complex numbers that depend on ${\bf x}$. We observe that for most of the $t \in I$, we have the estimate
$$|t-r_j| \geq \frac{|I|}{4n},\qquad j=1,\ldots,n.$$
Eliminating the exceptional intervals where this is not true,  we find a subset $I_{\bf x} \subset I$ with $|I_{\bf x}| > \frac{1}{2}|I|$, so that
$$|P_{\bf x} (t) | \geq \Big(\frac{|I|}{4n}\Big)^{n} |DG({\bf x})|,\qquad t \in I_{\bf x}.$$
Plugging this into \eqref{annette2} and applying Fubini's theorem, we conclude that
$$|S| \ge \frac{1}{(r-1)^n}\Big(\frac{|I|}{4n}\Big)^{n}\int_B  \int_{I_{\bf x}}    |DG({\bf x})|\,  dt d{\bf x}
\ge  \frac{1}{(r-1)^n} \Big(\frac{|I|}{4n}\Big)^{n} \frac{|I|}{2} |G(B)|.$$
Now, using our supposition \eqref{tit} and simplifying, recalling that $|I|=\lambda$, we obtain
$$1\ge \frac{1}{(r-1)^n} \Big(\frac{1}{4n}\Big)^{n} \frac{1}{2}C.$$
The $C$ appearing here is a constant multiple, depending only on $n$, $E$ and $\varepsilon$, of the constant appearing in \eqref{reptile}, which we take sufficiently large to obtain the desired contradiction. 
\end{proof}

\section{Proof of Theorem~\ref{mainThm} with $\lambda\ge \delta$}

The main difficulty in extending to the general case $\lambda\ge \delta$, is that  the condition 
$|T \cap S| \geq \lambda |T|$ is not semialgebraic. However, Theorem~\ref{almostsimple}  implies the following $\delta$-discretized version in which $S_{\delta}$ denotes the $\delta$-neighborhood of $S$ in $\mathbb{R}^{n+1}$.

\begin{thm}\label{prettysimple}
Let~$n,E\ge 1$ be integers and $\varepsilon>0$. Then there is a constant $C(n,E,\varepsilon)>0$ so that, for every set $\mathbb{T}$ of $\delta$-tubes in $\mathbb{R}^{n+1}$, pointing in  $\delta$-separated directions, 
\begin{equation*}\#\big(\big\{ T \in \mathbb{T} :  |T \cap S| \geq \lambda |T| \big\}\big) \leq C(n,E,\varepsilon) |S_\delta| \delta^{-n-\varepsilon} \lambda^{-n-1}\end{equation*}
whenever $S$ is a semialgebraic set, of complexity at most $E$, and $\lambda\ge \delta>0$.
\end{thm}

\begin{proof} Observe that if $|T \cap S| \ge \lambda |T|$, then there is a unit line segment
$\ell\subset T$ whose direction is that of $T$ and for which $|\ell \cap S| \geq \lambda$, where $|\ell \cap S|$ denotes one-dimensional Lebesgue measure. Now $\ell \cap S$ breaks into at most $C(E)$ connected components by B\'ezout's theorem. Thus,~$\ell \cap S$ contains a line segment of length $C(E)^{-1}\lambda$, and so $T_{\delta} \cap S_{\delta}$ contains a $C(E)^{-1}\lambda \times \delta$
tube in the direction of~$T$. Now we apply Theorem \ref{almostsimple}, with $S$ replaced by~$S_{\delta}$, to complete the proof. \end{proof}

This would be enough to prove the full theorem if we could bound $|S_{\delta}|$ in terms of $|S|$. Unfortunately, we do not not always have the appropriate bounds, so first we prove the same result for $S_{\eta}$ for any $\eta \ge \delta^{2n}.$

\begin{thm}\label{Netsisntfunny}
Let~$n,E\ge 1$ be integers and $\varepsilon>0$. Then there is a constant $C(n,E,\varepsilon)>0$ so that, for every set $\mathbb{T}$ of $\delta$-tubes in $\mathbb{R}^{n+1}$, pointing in  $\delta$-separated directions, 
$$\#\big(\big\{ T \in \mathbb{T} :  |T \cap S| \geq \lambda |T| \big\}\big) \leq C(n,E,\varepsilon) |S_\eta| \delta^{-n-\varepsilon} \lambda^{-n-1}.$$
whenever $S$ is a semialgebraic set, of complexity at most $E$, and $\lambda\ge \delta \ge \eta\ge \delta^{2n}>0$.
\end{thm}

\begin{proof}  We replace $\mathbb{T}$ by $\mathbb{T}_{\eta}$, a set of $\eta$-tubes pointing in $\eta$-separated directions. Taking no more than $(\delta/\eta)^n$  many $\eta$-tubes $V\subset T$, all intersecting in some ball of radius $\eta$, we can position them so that they capture a good proportion of the mass of $T\cap S$;  
\begin{equation}\label{tried}
c_n\lambda \delta^{n}\le \sum_{V\subset T} |V\cap S|.
\end{equation}
Writing $\mathbb{T}_\lambda=\{ T \in \mathbb{T} :  |T \cap S| \geq \lambda |T| \}$ and partitioning  into subsets~$\mathbb{V}_{\! k}$ of thin tubes $V$ that satisfy
\begin{equation}\label{try}
2^{-k}\lambda|V|\le |V \cap S| < 2^{-k+1}\lambda|V|,
\end{equation}
there must be a set $\mathbb{V}_{\! k}$ with large cardinality compared to~$\mathbb{T}_\lambda$.  Then we apply Theorem \ref{prettysimple} to this~$\mathbb{V}_{\!k}$, with~$\delta$ replaced
by $\eta$ and $\lambda$ replaced by $2^{-k}\lambda$, giving 
\begin{equation}\label{why}
\#\mathbb{V}_{\!k} \le C(n,E,\varepsilon) |S_\eta| \eta^{-n-\varepsilon} \lambda^{-n-1}2^{k(n+1)}.
\end{equation}

It is straightforward to find the subset  $\mathbb{V}_{\! k}$ with large cardinality compared to $\mathbb{T}_\lambda$.  We use the upper bound in \eqref{try}, to see that
\begin{equation}\label{firstly}\sum_{V\in \mathbb{V}_{\!k}} |V\cap S|< \sum_{T\in\mathbb{T}_\lambda}\sum_{V\subset T}2^{-k+1}\lambda|V|\le  2^{-k+1}\lambda\delta^n\#\mathbb{T}_\lambda,
\end{equation}
where the second inequality is because there are less than $(\delta/\eta)^n$ thin tubes in each fat tube. 
On the other hand, by summing \eqref{tried}, we have \begin{equation}\label{ref}
c_n\lambda \delta^n\#\mathbb{T}_\lambda\le \sum_{k\ge \log_2\lambda^{-1}} \sum_{V\in \mathbb{V}_{\!k}} |V\cap S|.\end{equation}
Comparing \eqref{firstly} and \eqref{ref}, we see that the summands with large $k$ contribute little and so, by the pigeonhole principal, \eqref{ref} must continue to hold for a single $\mathbb{V}_k$, with $k\le C(n)$,  losing only a factor of $\log_2\lambda^{-1}$. Using the upper bound of~\eqref{try} and recalling that $\lambda\ge \delta$, this yields
 \begin{equation*}\label{this}\#\mathbb{T}_\lambda\le C(n)\log_2\delta^{-1}\#\mathbb{V}_{\!k}2^{-k+1} \eta^{n}\delta^{-n},\end{equation*}
which can be combined with \eqref{why} to complete the proof.
\end{proof}

We are finally in a position to complete the proof of the full theorem.

\begin{proof}[Proof of Theorem \ref{mainThm}]
 We can suppose that  $|S| \geq \lambda\delta^n$, because otherwise there are
no tubes $T$ with $|T \cap S| \geq \lambda |T|$. Given that $|S_\eta|\le |S|+ |S_{\eta} \backslash  S|$, after applying Theorem~\ref{Netsisntfunny}, it would suffice to bound the measure of  the $\eta$-neighborhood of the boundary of $S$. By the Milnor--Thom theorem,  this is contained in the $\eta$-neighborhood of at most $C(n,E)$ hypersurfaces of degree at most $E$; see for example \cite[Theorem 9]{HRR}. Thus, we can apply Wongkew's lemma~\cite{Wo} to obtain
$$|S_{\eta} \backslash  S| \le C(n,E)\eta.$$
Taking $\eta=\delta^{{n+1}}\le|S|$, as we may, completes the proof.
\end{proof}

\end{document}